\documentclass[a4paper, 12pt]{amsart}
\usepackage[english]{babel}
\usepackage{amsmath,amsfonts,amssymb,amsthm,mathtools, mathabx, amscd} 
\usepackage[format=plain]{caption}
\usepackage[utf8]{inputenc}
\usepackage{enumitem}
\usepackage{setspace}
\setlist[enumerate]{leftmargin=7mm, label=\alph*)}
\usepackage[a4paper,asymmetric]{geometry}
\usepackage[sort]{cite}
\usepackage{hyperref}
\usepackage{comment}
\usepackage{xcolor}
\hypersetup{
     colorlinks = true,
     linkcolor = blue,
     anchorcolor = blue,
     citecolor = red,
     filecolor = blue,
     urlcolor = blue
     }
\usepackage{url}
\usepackage{euscript}

\theoremstyle{plain}
\newtheorem{theor}{Theorem}[section]
\newtheorem{prop}[theor]{Proposition}
\newtheorem{lem}[theor]{Lemma}

\newtheorem{cor}[theor]{Corollary}
\theoremstyle{definition}
\newtheorem{de}[theor]{Definition}
\newtheorem{const}[theor]{Construction}
\newtheorem{ex}[theor]{Example}
\newtheorem {re}[theor]{Remark}

\DeclareMathOperator{\Aut}{Aut}

\def\Ker{{\rm Ker}\,}

\def\GG{{\mathbb G}}

\def\KK{{\mathbb K}}

\def\ZZ{{\mathbb Z}}
\def\GA{{\mathbb G}_a}

\def\NN{{\mathbb N}}

\begin{document}
\date{}

\title[]{Isolated torus invariants and\\ automorphism groups of rigid varieties}

\author[Borovik]{Viktoriia Borovik}
\address[Borovik]{Osnabr\"uck University,
Fachbereich Mathematik/Informatik
Albrechtstr.~28a,
49076 Osnabrück, Germany.}
\email{vborovik@uni-osnabrueck.de}

\author[Gaifullin]{Sergey Gaifullin}
\address[Gaifullin]{Lomonosov Moscow State University, Faculty of Mechanics $\&$ Mathematics, Department of Higher Algebra, Leninskie Gory~1, 119991 Moscow, Russia;
$\&$ Higher School of Economics, Faculty of Computer Science, Pokrovsky Boulvard 11, 109028 Moscow, Russia.}
\thanks{The second author was supported by the RSF grant 22-41-02019.\ }
\email{sgayf@yandex.ru}


\maketitle
\thispagestyle{empty}

\begin{abstract}
Perepechko and Zaidenberg conjectured that the neutral component of the automorphism group of a rigid affine variety is a torus. We prove this conjecture for toric varieties and varieties with a torus action of complexity one. We also obtain a criterion for an $m$-suspension over a rigid variety to be rigid (for every rigid variety and every regular function). Additionally, we study the automorphism group of $m$-suspensions satisfying this criterion.  
\end{abstract}

\section{Introduction}

Let $\KK$ be an algebraically closed field of characteristic zero and $\GA$ be its additive group $(\KK,+)$. Suppose that $X$ is an affine algebraic variety over $\KK$. One~approach to investigate the group of regular automorphisms of $X$ is to study $\GA$-actions on~$X$. In general, the group $\Aut(X)$ of regular automorphisms of an affine variety $X$ is not a finite-dimensional algebraic group. In some cases this group is sufficiently large that it is not possible to give an explicit description of it. For instance, if $\dim X\geq 2$ and $X$ admits a nontrivial $\GA$-action, then $\Aut(X)$ is not~finite-dimensional. It is therefore natural to consider the class of rigid varieties. Recall that a variety is called {\it rigid} if it does not admit any nontrivial~$\GA$-actions. Similarly, we say that an algebra is rigid if it does not admit nontrivial~$\GA$-actions.

In some sense, the case of a rigid variety is the opposite case of a flexible variety, see \cite{AFKKZ}. For a flexible variety, there is a way to construct the set of flexible varieties from a given one. We can consider \emph{suspensions} over the variety $X$, that is, subvarieties in $\KK^2\times X$ given by $uv=f$, where $u$ and $v$ are coordinate functions on~$\KK^2$, and $f$ is a non-constant regular function on $X$. The suspension over a~flexible variety is flexible, cf. \cite{AKZ}.

A natural question arises whether we can obtain many rigid varieties from one by the same or similar construction. In~\cite[Definition~3.1]{Ga} the following generalization of the notion of suspensions is given.

\enlargethispage{\baselineskip}

\begin{de}
Let $f$ be a non-constant regular function on the affine variety $X$. We fix positive integers $k_1,\ldots, k_m$. By \emph{$m$-suspension} over $X$ we mean a subvariety in $\KK^m\times X$ cut by $y_1^{k_1}\cdots y_m^{k_m}-f$, where $y_1,\ldots,y_m$ are coordinate functions on~$\KK^m$. We denote this variety by $\mathrm{Susp}(X,f,k_1,\ldots,k_m)$. 
\end{de}

For $m=1$, the notion of $1$-suspension coincides with adjoining roots of $f$. In~\cite{FMJ}, a sufficient condition for a $1$-suspension to be rigid is obtained. This condition is expressed in terms of the number $k_1$, the variety $X$ and the function $f$. 
We are interested in the following question. For which $k_1,\ldots,k_m$ if $X$ is rigid, then $\mathrm{Susp}(X,f,k_1,\ldots,k_m)$ is rigid? It turns out that this is true if and only if~$\gcd(k_1,\ldots,k_m)=1$, see Theorem~\ref{perv}. In particular, we show that if $k\geq 1$, then there exists a rigid variety $X$ and a function $f\in\KK[X]$ such that the variety~$Y$ obtained by adjoining $k$-roots of $f$ is not rigid. 

It is clear that an $m$-suspension always admits an action of the $(m-1)$-dimensional algebraic torus~$T$.  We develop a technique related to the so-called isolated irreducible $T$-semi-invariants, see Definition~\ref{iisi}, which for a wide class of varieties allows us to describe the normalizer $N(T)$ of the torus $T$ in the automorphism group. When a variety is rigid, by~\cite[Theorem~1]{AG}, its automorphism group contains a unique maximal torus. Thus, the normalizer of this maximal torus coincides with the entire automorphism group. We therefore obtain a technique for investigating the automorphism group of a rigid variety. Applying this technique, we prove that in the case when an $m$-suspension $W:=\mathrm{Susp}(X,f,k_1,\ldots,k_m)$ with~$m\geq 2$ and $\gcd(k_1,\ldots,k_m)=1$ is an $m$-suspension over the rigid variety $X$ with only constant invertible functions, then~$\Aut(W)$ is a finite extension of the direct product $(\KK^\times)^{m-1}\times \Aut(X,f)$, where $\Aut(X,f)$ is the group of all automorphisms of $X$ multiplying $f$ by a constant, cf. Theorem~\ref{dva}. 

In \cite{Ra}, Ramanujam introduced the notion of a connected automorphism group. This notion allows us to define the neutral component of the automorphism group of an affine variety, see Definitions~\ref{dd} and \ref{ddd}. As we noted above, if the variety~$X$ is not rigid and its dimension is at least 2, then $\Aut(X)$ is always infinite-dimensional; moreover, its neutral component $\Aut(X)^0$ is also infinite-dimensional. Conversely, it is a natural conjecture that if $X$ is rigid, then $\Aut(X)^0$ is an algebraic group and even a torus of dimension at most $\dim(X)$, see~\cite[Conjecture~1.0.1]{PZ}. Using our technique, we prove this conjecture when $X$ is a (non-normal) toric affine variety and when $X$ is a rational normal affine variety with only constant invertible functions, finitely generated divisor class group and admitting a torus action of complexity one, see Theorem~\ref{vteo}. We also prove that if $\Aut(X)^0$ is a~torus, then~$\Aut(W)^0$ is a torus for the $m$-suspension $W=\mathrm{Susp}(X,f,k_1,\ldots,k_m)$ with $m\geq 2$ and $\gcd(k_1,\ldots,k_m)=1$, see Corollary~\ref{clcv}.

\subsection*{Acknowledgements.}
The second author is a Young Russian Mathematics award winner and would like to thank its sponsors and jury. 

\section{Preliminaries}
\noindent
In this section, we recall the key definitions and results that we will employ.

\subsection{Derivations}

Let $A$ be a $\KK$-domain. 

\begin{de}
A linear map $\partial\colon A\rightarrow A$ is called a \emph{derivation} if it satisfies Leibniz's rule: $\partial(ab)=a\partial(b)+b\partial(a)$.

A derivation is called {\it locally nilpotent (LND)} if for any $a \in A$ there exists $n\in \NN$ such that $\partial^n(a)=0$.
\end{de}
\noindent
The exponential map gives a correspondence between LNDs and subgroups in~$\mathrm{Aut}(A)$ isomorphic to the additive group $(\KK,+)$.
Let $F$ be an abelian group.

\begin{de}
An algebra $A$ is called {\it $F$-graded} if $$A=\bigoplus_{f\in F}A_f \text{ and } A_fA_g\subset A_{f+g}.$$
\end{de}

\begin{de}
A derivation $\partial\colon A\rightarrow A$ is called {\it $F$-homogeneous of degree $f_0\in F$} if for all $a\in A_f$ we have $\partial(a)\in A_{f+f_0}$.
\end{de}
\noindent
Each locally nilpotent derivation $\partial$ induces the function
\begin{align*}
\nu_\partial\colon A &\rightarrow \mathbb{Z}_{\geq 0}\cup\{-\infty\}    \\
a &\mapsto \min\{n\in\mathbb{Z}_{>0}\mid\partial^n(a)=0\}-1,\ \text{when}\ a\neq 0,
\end{align*}
and $\nu_\partial(0)=-\infty$. By \cite[Proposition~1.9]{Fr}, $\nu_\partial$ is a degree function, i.e.,
\begin{align*}
\nu_\partial(f+g) &\leq\max\{\nu_\partial(f),\nu_\partial(g)\},\\
\nu_\partial(fg) &=\nu_\partial(f)+\nu_\partial(g).
\end{align*}
Let $A$ be a finitely generated $\ZZ$-graded algebra.
\begin{lem}
Let $\partial$ be a derivation, then $\partial=\sum_{i=l}^k\partial_i$, where $\partial_i$ is a homogeneous derivation of degree $i$. 
\end{lem}
\begin{proof}
Let $a_1,\ldots,a_m$ be the generators of $A$. Then $\partial (a_j)=\sum_{i=l_j}^{k_j}b_i$, where $b_i\in A_i$. Take $l:=\mathrm{min}\{l_1,\ldots,l_m\}$, $k:=\mathrm{max}\{k_1,\ldots,k_m\}$. Using Leibniz's rule we get
$$\forall a\in A_j \colon \; \partial (a)\in\bigoplus_{i=l}^k A_{j+i}.$$ Thus, $\partial=\sum_{i=l}^k\partial_i$, where $\partial_i\colon A_j\rightarrow A_{j+i}$ is a linear map. Leibniz's rule for $\partial_j$ follows from Leibniz's rule for $\partial$.
\end{proof}

\begin{re}
Further, when we write $\partial=\sum_{i=l}^k\partial_i$, we will assume that $\partial_l\neq 0$ and~$\partial_k\neq 0$.
\end{re}

We need the following lemma. For the proof, we refer to~\cite{Re}.
\begin{lem}\label{fl}
Let $A$ be a finitely generated $\ZZ$-graded algebra, $\partial\colon A\rightarrow A$ be an~LND. Assume $\partial=\sum_{i=l}^k\partial_i$, where $\partial_i$ is the homogeneous derivation of degree $i$. Then $\partial_l$ and $\partial_k$ are LNDs.
\end{lem}
\begin{cor}\label{flz}
If $A$ admits an LND, then $A$ admits a $\ZZ$-homogeneous LND. 
\end{cor}

\begin{de}
An affine algebraic variety $X$ is called {\it rigid} if its algebra of regular functions $\KK[X]$ does not admit any nonzero LNDs.
\end{de}

By~\cite[Theorem~1]{AG}, the automorphism group of a rigid variety contains a unique maximal algebraic torus.

\subsection{Neutral component of an automorphism group}
Let us define the connected component $\Aut(X)^0$ of the group $\Aut(X)$. We are following \cite{Ra}.

\begin{de}\label{dd} The family $\{\varphi_b \,|\, b\in B\}$ of automorphisms of a variety $X$, where the parameterizing set $B$ is an algebraic variety, is {\it an algebraic family} if the map~$B\times X\rightarrow X$ given by $(b,x)\mapsto \varphi_b(x)$ is a morphism.
\end{de}

\begin{de}\label{ddd} \emph{The connected component} $\Aut(X)^0$ of the group $\Aut(X)$ is a~subgroup of
automorphisms that can be included in the algebraic family $\{\varphi_b \,|\, b\in B\}$ with an irreducible
variety~$B$ as a base such that $\varphi_{b_0}=\mathrm{id}_X$ for some $b_0\in B$.
\end{de}
\noindent
It is easy to check that $\Aut(X)^0$ is indeed a subgroup, see~\cite{Ra}.
If $G$ is an algebraic group and $G\times X\rightarrow X$ is a regular action, then we can take $B = G$ and consider the algebraic family $\{\varphi_g \,|\, g\in G\}$, where $\varphi_g(x)=gx$. Thus, any automorphism given by an element in $G$ is included in $\Aut(X)^0$. In particular, every $\GG_a$- and every $\GG_m$-subgroup is contained in $\Aut(X)^0$, where $\GG_m \simeq (\KK^{\times}, \cdot)$.


\subsection{Trinomial varieties}\label{tv}

We are giving a definition of a trinomial variety according to~\cite{H-W}.

\begin{const}\label{odyn} \cite[Construction 1.1]{H-W}. 
Suppose we are given positive integers~$r$ and~$n$, a non-negative integer $m$ and $q \in \left\{ 0, 1 \right\}$. Let us fix the partition~$n = n_{q} + \ldots+ n_{r}$ of~$n$ into positive integer summands. 
Consider a polynomial algebra $B$ in $m+n$ variables. We denote these variables by $T_{ij}$ and~$S_k$:
$$
B= \KK \left[ {T_{ij}, S_{k} \mid q \leq i \leq r, \!\ 1 \leq j \leq n_{i}, \!\ 1 \leq k \leq m} \right].
$$
For each $i = q, \ldots, r$ we fix a tuple of positive integers $l_{i} = \left( {l_{i1},  \ldots , l_{in_{i}}} \right)$. So we can consider the following monomial:
$$
	T_{i}^{l_{i}} = T_{i1}^{l_{i1}}  \cdots T_{in_{i}}^{l_{in_{i}}} \in B. 
$$
We now define {\it a trinomial algebra} $R(A)$, which we construct from some data $A$.  These data are different for two types of trinomial algebras. 
\par \emph{Type 1.} Let $q = 1$ and $A = (a_{1},  \ldots, a_{r})$ be a list of distinct elements of~$\KK$.  Set~$I = \left\{ {1, \ldots, r - 1} \right\}$, then take
\begin{equation*}
	g_{i} := T_{i}^{l_{i}} - T_{i+1}^{l_{i+1}} - (a_{i+1} - a_{i}) \in B \text{ for } i \in I.
\end{equation*}
\par \emph{Type 2.} Let $q = 0$ and take a matrix $A \in {\rm Mat}_{2\times (r+1)}(\KK)$,
\begin{equation*}
	A = \begin{bmatrix}
		a_{10} \ a_{11} \ a_{12} \ \cdots\  a_{1r} \\
		a_{20} \ a_{21} \ a_{22} \ \cdots \ a_{2r} \\
	\end{bmatrix},
\end{equation*}
 with pairwise linearly independent columns. Let us set $I = \left\{ {0,  \ldots , r - 2} \right\}$ and 
\begin{equation*}
	g_{i} := \det \begin{bmatrix}
		T_{i}^{l_{i}} \ T_{i+1}^{l_{i+1}} \ T_{i+2}^{l_{i+2}} \\
		a_{1i} \ a_{1i+1} \ a_{1i+2}\\
		a_{2i} \ a_{2i+1} \ a_{2i+2}
	\end{bmatrix} \in B \text{ for } i \in I.
\end{equation*}
For both types we define $R(A)$ to be a factor algebra $B/\langle g_{i}\, |\, i \in I\rangle$.
\end{const}
\begin{de} The variety $X(A) = \mathrm{Spec}(R(A))$ is called a {\it trinomial variety}.  The type of a trinomial variety is the type of the corresponding trinomial algebra.\end{de} 

\begin{re}
The dimension of $X(A)$ equals $m+n-r+1$.
\end{re}

We can define an action of $\mathbb{T}\cong (\KK^\times)^{n+m-r}$ of complexity one on the variety~$X(A)$. Recall that \emph{the complexity} of the action of the reductive group $G$ is the codimension of the generic orbit of its Borel subgroup. In the case when $G$ is a torus acting effectively on the variety $X$, the complexity is simply $\dim(X)-\dim(G)$. It~is known that the character group of an algebraic torus is a free abelian group, and setting an action of a torus of dimension $k$ on $X$ is equivalent to setting $\mathbb{Z}^k$-grading on the algebra of regular functions $\KK[X]$. We define a $\mathbb{Z}^{m+n-r}$-grading on~$R(A)$ as follows. Let us assume that the variable $T_{ij}$ has degree $w_{ij}$ and the variable~$S_k$ has degree $v_k$. These degrees with commutativity relations, $\deg T_i^{l_i}=\deg T_j^{l_j}$, and~$\deg T_i^{l_i}=0$ for type 1, generate the group $\mathbb{Z}^{m+n-r}$, see~\cite{H-W} for details.  

\begin{de}\label{ddppd}
A $\mathbb{Z}^r$-grading on the $\KK$-algebra $B$ is called {\it pointed} if $B_0=\KK$ and for each $a\neq 0\in\mathbb{Z}^r$ one of the spaces $B_a$ and $B_{-a}$ is zero. (Note that the second condition is equivalent to saying that the weight cone does not include a~line.)

If $\mathbb{Z}^r$-grading is pointed, we call the corresponding torus action pointed. 
\end{de}
\begin{re}
 If $X(A)$ is of type 2, then the $\mathbb{T}$-action defined above is pointed.    
\end{re}

\textit{}
\section{Normalizer of a torus}

Consider a torus $T$ acting efficiently on an affine variety $X$.  Then we can consider~$T$ as a subgroup of $\Aut(X)$. We denote by $N(T)$ the normalizer of $T$ in~$\Aut(X)$, i.e.,
$$
N(T)=\{g\in \Aut(X)\mid g\circ t\circ g^{-1}\in T \text{ for all } t\in T\}.
$$
Recall that a regular function $f\in\KK[X]$ is called {\it semi-invariant with respect to $T$} if for all $t\in T$ we have $t\cdot f=\lambda f$ for some $\lambda=\lambda(t)\in \KK^\times$. A character $\lambda\colon T\rightarrow \KK^\times$ is called {\it a weight} of $f$. We call a semi-invariant $f$ {\it  irreducible} if it is irreducible as an element of $\KK[X]$. It is easy to show that a $T$-semi-invariant $f$ is irreducible if it is not invertible and cannot be decomposed into the product of two non-invertible $T$-semi-invariants. Let us denote the lattice of $T$-characters by $M$. 

\begin{de}\label{iisi}
We call an irreducible $T$-semi-invariant $f$ of weight $\omega$ {\it isolated} if there exists a linear function $\alpha$ on $M_\mathbb{Q}=M\otimes_{\mathbb{Z}}\mathbb{Q}$ such that
\begin{enumerate}
\item[1.] $\alpha(\omega)>0$;
\item[2.] if $\alpha(\omega')>0$ for the weight $\omega'$ of an irreducible $T$-semi-invariant $h$, then $h=\gamma f$, where $\gamma$ is an invertible $T$-semi-invariant;
\item[3.] if $\omega''$ is the weight of an invertible $T$-semi-invariant, then $\alpha(\omega'')=0$.
\end{enumerate}
We say that $\alpha$ is an $f$-{\it separating} function.
\end{de}

\begin{lem}\label{finas}
There exist finitely many isolated irreducible $T$-semi-invariants up to scaling by elements of $\KK[X]^\times$.								
\end{lem} 
\begin{proof}
The algebra $\KK[X]$ is finitely generated. Hence, there exists a finite system $S$ of $T$-semi-invariants generating $\KK[X]$. We may assume that $S$ consists of irreducible and invertible semi-invariants.

Let $f$ be an isolated irreducible $T$-semi-invariant, and let $\alpha$ be an $f$-separating function. 
Each $T$-semi-invariant has a~weight equal to the sum of the weights of the elements of $S$. If $\alpha$ has non-positive values on the weights of all $s\in S$, then it has non-positive values on the weights of all semi-invariants. This is a contradiction with $\alpha(\omega)>0$, where $\omega$ is the weight of $f$. Hence, there exists an element $s\in S$ such that $\alpha$ is positive on the weight of $s$. Therefore, $s=\gamma f$ for some $\gamma\in\KK[X]^\times$.
Since $S$ is finite, we get the result.
\end{proof}
\begin{de}
    We call $[f] := \{h \in \KK[X] \mid h = \gamma f \text{ for some } \gamma \in\KK[X]^\times\}$ \emph{the association class of the isolated irreducible $T$-semi-invariant $f$}.
\end{de}
Thus, Lemma~\ref{finas} states that there are finitely many association classes of isolated irreducible $T$-semi-invariants.

A natural question is then how to prove that the function $f$ is an isolated irreducible semi-invariant. The following proposition shows that if there exists an~$f$-separating function on generators, then $f$ is irreducible.

\begin{prop}\label{dsi}
Let $S$ be the set of $T$-semi-invariant generators of $\KK[X]$. 
\begin{enumerate}
    \item[$(i)$] We fix a non-invertible element $f\in S$ and denote its weight by~$\omega$. Suppose that there exists a linear function $\alpha$ on $M_\mathbb{Q}$ such that $\alpha(\omega)>0$ and $\alpha(\omega')\leq 0$ for all weights $\omega'$ of the other generators $s\in S$. Then $f$ is an irreducible isolated $T$-semi-invariant. 
    \item[$(ii)$]  Every irreducible isolated $T$-semi-invariant has the form $\gamma f$ for some $f$ satisfying $(i)$ and $\gamma\in\KK[X]^\times$.
\end{enumerate}
\end{prop}
\begin{proof}
$(i)$. Suppose that $f$ is reducible. Let $f=f_1 \cdots f_m$ be a decomposition into irreducible semi-invariants. Then there exists $f_k$ such that for its weight $\omega_k$ we have $\alpha(\omega_k)>0$. But $f_k$ is a polynomial in $s\in S$. Since $\alpha(\omega_k)>0$, then $f \mid f_k$. This is possible only if $f=f_k$, i.e., $f$ is irreducible.

Let us prove that $f$ is isolated. Take an irreducible element $h$ such that $\alpha(h)>0$. Then, as stated above, $f\mid h$. Then $h=\gamma f$ for some $\gamma\in \KK[X]^\times$. 

$(ii)$. In the proof of the Lemma~\ref{finas}, it is shown that every irreducible isolated $T$-semi-invariant has the form $\gamma f$ for some $f\in S$. But then $f$ is also an irreducible isolated $T$-semi-invariant. Hence, $f$ satisfies $(i)$.
\end{proof}

Note that Proposition~\ref{dsi} gives an explicit algorithm for finding all irreducible isolated $T$-semi-invariants.
If $f$ is an isolated $T$-semi-invariant and~$\varphi\in N(T)$, then~$\varphi(f)$ is also an isolated $T$-semi-invariant. Thus, each element of $N(T)$ permutes the association classes of isolated semi-invariants. 

Let $\mathcal{A}\subseteq \KK[X]$ be a subalgebra generated by all invertible functions and isolated semi-invariants, and let $\mathcal{B}\subseteq \mathcal{A}\subseteq \KK[X]$ be a subalgebra generated by all invertible functions.  Then $\mathcal{A}$ is a $N(T)$-invariant subalgebra. Both subalgebras $\mathcal{A}$ and $\mathcal{B}$ are finitely generated, so we can consider their spectra $\mathrm{Spec}\,\mathcal{A}=Y$, $\mathrm{Spec}\,\mathcal{B}=Z$. The subalgebra $\mathcal{B}$ is $\Aut(X)$-invariant in $\KK[X]$ and $\Aut(Y)$-invariant in $\KK[Y]$.

Consider the restrictive homomorphism $\Psi\colon \Aut(X)\rightarrow \Aut(Y)$, $\Psi(\varphi)=\varphi|_{\mathcal{A}}$. The image $\widehat{T}$ of the torus $T$ under the action of $\Psi$ is an algebraic subtorus in~$\Aut(Y)$. Clearly, the isolated irreducible $\widehat{T}$-semi-invariants coincide with the isolated irreducible $T$-semi-invariants. Thus we have $\Psi( N(T))\subseteq N(\widehat{T})$. We can also consider the restrictive homomorphism $\Phi\colon \Aut(Y)\rightarrow \Aut(Z)$, $\Phi(\varphi)=\varphi|_{\mathcal{B}}$. Thus, we have the following chain of homomorphisms
$$
N(T)\xrightarrow\Psi N(\widehat{T})\xrightarrow\Phi \Aut(Z).
$$
The variety $Z$ is toral, i.e., $\KK[Z]$ is generated by invertible elements. The automorphism groups of toral varieties are studied in \cite{ShT}. There it is proved that the neutral component of $\Aut(Z)^0$ is a torus and $\Aut(Z)$ is a discrete extension of~$\Aut(Z)^0$.  We use the homomorphism $\Phi$ to prove the following theorem.
\begin{theor}\label{isol}
\begin{enumerate}
    \item[$(i)$] The neutral component $N(\widehat{T})^0$ is a torus. 
    \item[$(ii)$] If $X$ does not admit non-constant invertible functions, then $N(\widehat{T})$ is a finite extension of a torus.  
\end{enumerate}
\end{theor}
\begin{proof}
$(i)$.  Since we know that the image of $\Phi$ is contained in the automorphism group of a toral variety $Z$, it remains to prove that the neutral component of its kernel is a torus. As we mention above, the group $N(\widehat{T})$ permutes the association classes of the isolated $\widehat{T}$-semi-invariants. Thus we obtain a homomorphism~$\xi\colon N(\widehat{T})\rightarrow S_m$ to the symmetric group $S_m$, where $m$ is the number of the association classes of the isolated $\widehat{T}$-semi-invariants in $\KK[Y]$. 

We restrict $\xi$ to the homomorphism $\overline{\xi}\colon \Ker\Phi\rightarrow S_m$. Clearly,  
$(\Ker\Phi)^0$ is contained in $\mathrm{Ker}\,\overline{\xi}$. Take~$\varphi\in \mathrm{Ker}\,\overline{\xi}$, then for every isolated $\widehat{T}$-semi-invariant~$f$, we have~$\varphi(f)=\alpha f$, where $\alpha\in \KK[Y]^\times$. Also, $\varphi$ does not change the invertible functions. Thus, we get the embedding $\mathrm{Ker}\,\overline{\xi}\hookrightarrow (\KK[Y]^\times)^m$. By~\cite{R}, the group~$\KK[Y]^\times/ \KK^\times$ is a free abelian group with finite number of generators. Therefore, the neutral component of $(\KK[Y]^\times)^m$ is $(\KK^\times)^m$. Hence, the neutral component of~$\mathrm{Ker}\,\overline{\xi}$ is contained in $(\KK^\times)^m$, i.e., it is a torus. This gives the result.

$(ii)$. If $X$ does not admit non-constant invertible functions, then $Z$ is a point, hence $N(\widehat{T})=\Ker\Phi$.  As above $\mathrm{Ker}\,\overline{\xi}\hookrightarrow (\KK[Y]^\times)^m=(\KK^\times)^m$. Thus, $\mathrm{Ker}\, \xi=\mathrm{Ker}\,\overline{\xi}$ is a finite extension of the torus. Therefore, $N(\widehat{T})=\Ker\Phi$ is a finite extension of the torus.
\end{proof}

We can use Theorem~\ref{isol} to investigate $N(T)$ in the case when $\KK[X]$ is an integer extension of $\mathcal{A}$.

\begin{cor}\label{cledct}
Suppose $\KK[X]$ is an integer extension of $\mathcal{A}$. Then
\begin{enumerate}
    \item[$(i)$] The neutral component $N(T)^0$ is a torus. 
    \item[$(ii)$] If $X$ does not admit non-constant invertible functions, then $N(T)$ is a finite extension of a torus.  
\end{enumerate}
\end{cor}
\begin{proof}
Since $\KK[X]$ is an integer extension of $\mathcal{A}=\KK[Y]$, the kernel of~$\Psi$ is a finite group. Therefore, $N(T)$ is a finite extension of $\mathrm{Im}\,\Psi\subseteq N(\widehat{T})$.
\end{proof}
\begin{re}
One can explicitly find the maximal torus $\widetilde{T}$ in $N(\widehat{T})$. Let us choose generators $f_1,\ldots, f_k$ of $\mathcal{A}$ such that each $f_j$ is an invertible function or an isolated semi-invariant. Then $\mathcal{A}=\KK[f_1,\ldots, f_k]/I$. The torus $\widetilde{T}$ is the stabilizer $I$ in the $k$-dimensional torus acting on each $f_i$ by multiplication by a constant. It is often possible to find $N(T)$ explicitly using this idea.
\end{re}
\begin{ex}
Let $X= \mathbb{V}(ab-cd-1) \subseteq \mathbb{A}^4$, $X\cong\mathrm{SL}(2, \KK)$. And let $T\cong(\KK^\times)^2$ act on $X$ by the rule
$$
(t_1,t_2)\cdot (a,b,c,d)=(t_1a, \, t_1^{-1}b, \, t_2 c, \, t_2^{-1}d).
$$
The functions $a,b,c$ and $d$ form a system of semi-invariant generators $\KK[X]$ with weights 
$$\omega_a=(1,0), \quad \omega_b=(-1,0), \quad \omega_c=(0,1), \quad \omega_d=(0,-1).$$
Using the functions 
$$\alpha_a\colon (x,y)\mapsto x, \; \alpha_b\colon (x,y)\mapsto -x, \; \alpha_c\colon (x,y)\mapsto y, \; \alpha_d\colon (x,y)\mapsto -y,$$
by Proposition~\ref{dsi} we deduce that $a,b,c$ and $d$ are isolated semi-invariants. In this case, $Y=X$ and we have $N(T)=N(\widehat{T})$. Since $X$ does not admit non-constant invertible functions, by Theorem~\ref{isol} the group $N(T)$ is a finite extension of the torus. It is easy to show that $T$ is a maximal torus in $\Aut(X)$ and $\Aut(X)\cong \mathrm{D}_4 \rtimes T$, where $\mathrm{D}_4$ is a dihedral group.
\hfill $\diamond$ 
\end{ex}
\begin{ex}
Let $X=\mathbb{V}(x_1x_2x_3x_4-y^2)\subseteq \mathbb{A}^{5}$. Consider $T\cong(\KK^\times)^3$ acting on~$X$ by the rule
$$
(t_1,t_2,t_3)\cdot (x_1,x_2,x_3,x_4,y)
=(t_1t_2x_1,\, t_1t_2^{-1}x_2,\, t_1t_3x_3,\, t_1t_3^{-1}x_4,\, t_1^2y).
$$
The functions $x_1,x_2,x_3,x_4$ and $y$ form a system of semi-invariant generators of~$\KK[X]$ with weights 
\begin{align*}
\omega_{y}=(2,0,0), \quad \omega_{x_1}&=(1,1,0), \quad \omega_{x_2}=(1,-1,0), \\ 
 \omega_{x_3}&=(1,0,1), \quad \omega_{x_4}=(1,0,-1). 
\end{align*}
Using the functions 
\begin{align*}
\alpha_1\colon (x,y,z) & \mapsto y, \quad \alpha_2\colon (x,y,z)\mapsto -y,  \\
\alpha_3\colon (x,y,z) & \mapsto z, \quad \alpha_4\colon (x,y,z)\mapsto -z,
\end{align*}
it follows from Proposition~\ref{dsi} that $x_1,x_2,x_3$ and~$x_4$ are isolated semi-invariants. Since $\KK[X]$ is an integer extension of $\KK[x_1,x_2,x_3,x_4]$, by Theorem~\ref{isol} we have that $N(T)$ is a finite extension of the torus. It can be shown that $N(T)\cong\mathrm{S}_4\rtimes \mathbb{T}$, where $\mathrm{S}_4$ acts by permuting $\{x_1,x_2,x_3,x_4\}$, and $\mathbb{T}\cong(\KK^\times)^4$ acts by 
$$
(t_1,t_2,t_3,t_4)\cdot (x_1,x_2,x_3,x_4,y)= (t_1^2x_1,\, t_2^2x_2,\, t_3^2x_3,\, t_4^2x_4,\, t_1t_2t_3t_4y).
$$
This example shows that $T$ can be a non-maximal torus in $N(T)$.
\hfill $\diamond$ 
\end{ex}
Note that if $X$ is rigid, then, as we mentioned above, there exists a unique algebraic torus $T$ in $\Aut(X)$. Since $T$ is unique, it is normal. Hence, $N(T)$ equals to the whole group~$\Aut(X)$. Thus, we obtain the following corollary.

\begin{cor}\label{nmmn}
Let $X$ be a rigid variety and let $T$ be the maximal torus in $\Aut(X)$. Suppose that $\KK[X]$ is an integer extension of the subalgebra $\mathcal{A}$ generated by all invertible functions and isolated $T$-semi-invariants in~$\KK[X]$. Then 
\begin{enumerate}
    \item[$(i)$] The neutral component $\Aut(X)^0$ coincides with $T$. 
    \item[$(ii)$] If $X$ does not admit non-constant invertible functions, then $\Aut(X)$ is a finite extension of $T$.  
\end{enumerate}
\end{cor}

\section{Rigidity of $m$-suspensions}

Let us fix a tuple of positive integers $(k_1,\ldots,k_n)$. A natural question to ask is when an $m$-suspension $W=\mathrm{Susp}(X,f,k_1,\ldots,k_m)$ over a variety $X$ is rigid. 
In this section, we investigate for which tuples of numbers $(k_1,\ldots, k_n)$ the $m$-suspension $W=\mathrm{Susp}(X,f,k_1,\ldots,k_m)$ is rigid for every rigid variety $X$ and for every $f\in \KK[X]$. 
\begin{re}
By definition, $f$ is a non-constant function. But throughout the whole paper, all the statements and proofs are true for the case when $f$ is a constant.
\end{re}
\subsection{Roots adjoining}
In this subsection, we study the case of adjoining a root of a regular function of a rigid variety. This is the case of $1$-suspension. Our goal is to prove the following proposition.

\begin{prop}\label{dk}
Let $n\in \ZZ, n\geq 2$. There exist a rigid variety $X$ and a~regular function $f\in\KK[X]$ such that the algebra $\mathcal{C}=\KK[X][y]/\langle y^n-f\rangle$ is not rigid.
\end{prop}
\begin{proof}
Consider a prime divisor $p$ of $n$. First assume that $p\geq 3$. Let us denote by~$\varepsilon_1,\ldots,\varepsilon_p\in\KK$ all roots of unity of degree $p$. In $(p+3)$-dimensional affine space, consider a variety $V$ given by the following two equations:
\begin{align*}
\prod_{i=1}^p\left(x_0+\varepsilon_i x_1y+\varepsilon_i^2 x_2y^2+\ldots+\varepsilon_i^{p-1} x_{p-1}y^{p-1}\right)  =z^2 \; \text{ and } \;
y^{3p}+w^3  =1.
\end{align*}
Let us denote the left-hand side of the first equation by 
$$F(x_0,\ldots,x_{p-1},y):=\sum f_{i_0,\ldots,i_{p-1},j} \cdot x_0^{i_0}\cdots x_{p-1}^{i_{p-1}} \cdot y^j.$$ 
Each coefficient in $f_{i_0,\ldots,i_{p-1},j}$ is a homogeneous symmetric polynomial in $\varepsilon_1,\ldots,\varepsilon_p$ of degree $j$ with integer coefficients. Each symmetric polynomial is a polynomial in elementary symmetric polynomials $\sigma_1,\ldots,\sigma_p$. It's easy to see that 
\begin{align*}
\sigma_1(\varepsilon_1,\ldots,\varepsilon_p) & =0, \; \ldots, \; \sigma_{p-1}(\varepsilon_1,\ldots,\varepsilon_p)=0, \\  
\sigma_{p}(\varepsilon_1,\ldots,\varepsilon_p) & =(-1)^{p+1}=1.
\end{align*}
Therefore, if $j$ is not a multiple of $p$, then $f_{i_0,\ldots,i_{p-1},j}=0$. That is, $$F(x_0,\ldots,x_{p-1},y)=G(x_0,\ldots,x_{p-1},y^p).$$

Let us prove that the variety $V$ is not rigid. For this purpose we define a nonzero locally nilpotent derivation $\partial$ by the rule 
$$
\begin{cases}
\partial(y)=\partial(w)=0,\\
\partial(z)=y^{p-1}\prod_{i=2}^p \left(x_0+\varepsilon_i x_1y+\varepsilon_i^2 x_2y^2+\ldots+\varepsilon_i^{p-1} x_{p-1}y^{p-1}\right),\\
\partial\left(x_0+\varepsilon_1 x_1y+\varepsilon_1^2 x_2y^2+\ldots+\varepsilon_1^{p-1} x_{p-1}y^{p-1}\right)=2zy^{p-1},\\
\partial\left(x_0+\varepsilon_i x_1y+\varepsilon_i^2 x_2y^2+\ldots+\varepsilon_i^{p-1} x_{p-1}y^{p-1}\right)=0 \text{ for } i\geq 2.
\end{cases}
$$
These conditions give linear equations on $\partial(x_0), \partial(x_1)y, \ldots, \partial(x_{p-1})y^{p-1}.$ The determinant of this system is 
$$
\begin{vmatrix}
1&\varepsilon_1&\ldots&\varepsilon_1^{p-1}\\
\vdots&\vdots&\vdots&\vdots\\
1&\varepsilon_p&\ldots&\varepsilon_p^{p-1}
\end{vmatrix}\neq 0.
$$
Hence, these conditions admit a unique solution $\partial\neq 0$.
It is easy to see that 
$$\partial(F)=\partial(z^2)\ \text{ and }  \ \partial(y^{3p}+w^3)=\partial(1)=0.$$ 
That is, $\partial$ is a derivation on $\KK[V]$.     
One can also check that $\partial^2(z)=0$ and $\partial^3\left(x_0+\varepsilon_1 x_1y+\varepsilon_1^2 x_2y^2+\ldots+\varepsilon_1^{p-1} x_{p-1}y^{p-1}\right)=0$. This implies that~$\partial$ is an LND. 

If we set $y=u^{\frac{n}{p}}$, we obtain a new variety $W$. Since $\partial(y)=0$, the derivation $\partial$ induces a nonzero locally nilpotent derivation on $\KK[W]$. 

Now let us put $s=y^p=u^n$. We obtain a new variety $X$ given by 
$$
G(x_0,\ldots,x_{p-1},s)=z^2 \ \text{ and } \ 
s^3+w^3=1.
$$
Let us prove that $X$ is rigid. Suppose that there exists a locally nilpotent derivation~$\delta\neq 0$ on $\KK[X]$. 
We define a $\mathbb{Z}$-grading on $\KK[X]$ by setting $\deg(x_i)=2$, $\deg(z)=p$, $\deg(s)=\deg(w)=0$. Since $\KK[X]$ admits a nonzero locally nilpotent derivation, it admits a nonzero homogeneous locally nilpotent derivation $\rho$, see Corollary~\ref{flz}. 
Since $\rho(s^3+w^3)=0$, we get $\rho(s)=\rho(w)=0$, see \cite[Theorem~2.47]{Fr}. Hence, $\rho$ induces an LND on $\KK[V]$, which we will also denote by $\rho$. 

If $\deg(\rho)$ is even, then $z\mid \rho(z)$. Therefore, $\rho(z)=0$. But $\mathrm{Ker}\,\rho$ is factorially closed, cf. \cite[Principle~1]{Fr}. So for every $i$, we have
$$\rho\left(x_0+\varepsilon_i x_1y+\varepsilon_i^2 x_2y^2+\ldots+\varepsilon_i^{p-1} x_{p-1}y^{p-1}\right)=0.$$ 
It follows that $\rho=0$. This gives a contradiction.

Now let $\deg(\rho)$ be odd. Then $z\mid\rho^2(z)$, that is $\rho^2(z)=0$. Hence, if $\rho(z)\neq 0$, then $\nu_\rho(z)=1$. Therefore $\nu_\rho(F)=\nu_\rho(z^2)=2$. 
Since 
$$\deg\left(x_0+\varepsilon_i x_1y+\varepsilon_i^2 x_2y^2+\ldots+\varepsilon_i^{p-1} x_{p-1}y^{p-1}\right)=2$$ 
and $2\nmid\deg(\rho)$, we have $$z\mid \rho\left(x_0+\varepsilon_i x_1y+\varepsilon_i^2 x_2y^2+\ldots+\varepsilon_i^{p-1} x_{p-1}y^{p-1}\right).$$
 That is,
 \begin{align*}
   \text{either} \quad \nu_\rho\left(x_0+\varepsilon_i x_1y+\varepsilon_i^2 x_2y^2+\ldots+\varepsilon_i^{p-1} x_{p-1}y^{p-1}\right) & =0  \\
   \text{or} \quad \nu_\rho\left(x_0+\varepsilon_i x_1y+\varepsilon_i^2 x_2y^2+\ldots+\varepsilon_i^{p-1} x_{p-1}y^{p-1}\right)  & \geq 2.
 \end{align*}
But $\nu_\rho(F)=2$. 
 Therefore, there is $k$ such that 
 $$\nu_\rho\left(x_0+\varepsilon_k x_1y+\varepsilon_k^2 x_2y^2+\ldots+\varepsilon_k^{p-1} x_{p-1}y^{p-1}\right)=2$$
 and for all $i\neq k$ we have 
$$\rho\left(x_0+\varepsilon_i x_1y+\varepsilon_i^2 x_2y^2+\ldots+\varepsilon_i^{p-1} x_{p-1}y^{p-1}\right)=0.$$
Let us consider the conditions 
$$
\begin{cases}
\rho\left(x_0+\varepsilon_i x_1y+\varepsilon_i^2 x_2y^2+\ldots+\varepsilon_i^{p-1} x_{p-1}y^{p-1}\right)=0 \ \text{ for } \ i\neq k,\\
\rho\left(x_0+\varepsilon_k x_1y+\varepsilon_k^2 x_2y^2+\ldots+\varepsilon_k^{p-1} x_{p-1}y^{p-1}\right)=h\in\KK[V]
\end{cases}
$$
as a system of linear equations on $\rho(x_0), \rho(x_1)y, \ldots, \rho(x_{p-1})y^{p-1}$ with the determinant of this system equal to 
$$
\begin{vmatrix}
1&\varepsilon_1&\ldots&\varepsilon_1^{p-1}\\
\vdots&\vdots&\vdots&\vdots\\
1&\varepsilon_p&\ldots&\varepsilon_p^{p-1}
\end{vmatrix}\neq 0.
$$
Let us solve this system using Cramer's rule. We see that $\rho(x_j)y^j=\lambda_j h$ for some~$\lambda\in \KK$. Suppose $\lambda_l\neq 0$. Then 
$$\rho(x_j)=\frac{\lambda_j h}{y^j}=\frac{\lambda_j \lambda_l h}{\lambda_l y^j}=\frac{\lambda_j y^l\partial(x_l)}{\lambda_l y^j}=\frac{\lambda_j}{\lambda_l}y^{l-j}\rho(x_l).$$
But $\rho(x_j), \rho(x_l)\in\KK[X]$ and $y^{l-j}\notin \KK(X)$ for $l\neq j$. This implies $\rho(x_j)=0$ for each~$j\neq l$.
Since for $i\neq k$ we have 
$$\rho\left(x_0+\varepsilon_i x_1y+\varepsilon_i^2 x_2y^2+\ldots+\varepsilon_i^{p-1} x_{p-1}y^{p-1}\right)=0,$$ 
we obtain $\rho=0$, which gives a contradiction.

We proved that $X$ is a rigid variety. But the algebra $\KK[W]$ can be obtained from the algebra $\KK[X]$ by adjoining a root of $s$ of degree $n$. And $W$ is not rigid. This proves the proposition in the case $p \geq 3$.

We now consider the case $p=2$, i.e., $n=2k$. By \cite[Theorem~9.2]{FMJ}, the variety~$X= \mathbb{V}(x^2+y^2s^3+z^3)$ is rigid. However, by \cite[Remark~9.1]{FMJ}, the variety~$W=\mathbb{V}(x^2+y^2u^{6k}+z^3)$ is not rigid. We refer to \cite[Theorem~2]{G} in both~cases.
\end{proof}

\subsection{General case of $m$-suspensions}

In this subsection, we obtain a criterion for an $m$-suspension $\mathrm{Susp}(X,f,k_1,\ldots,k_m)$ to be rigid for each $X$ and each $f$. 

We need the following lemma, see \cite[Lemma~11]{G}.
\begin{lem}\label{nod}
 Assume an $m$-suspension $W=\mathrm{Susp}(X,f,k_1,\ldots,k_m)$ is not rigid. Let $d=\gcd(k_1, \ldots, k_m)$. Then $V=\mathrm{Susp}(X,f,d)$ is not rigid.
\end{lem}

\begin{theor}\label{perv}
Fix $m$ positive integers $k_1,\ldots, k_m$ and let $d=\gcd(k_1, \ldots, k_m)$. 
\begin{enumerate}
    \item[$(i)$] If $d=1$, then for every rigid variety $X$ the variety $W=\mathrm{Susp}(X,f,k_1,\ldots,k_m)$ is also rigid.
    \item[$(ii)$] If $d>1$, then there exists such  rigid variety $X$ that $W=\mathrm{Susp}(X,f,k_1,\ldots,k_m)$ is not rigid.
\end{enumerate}
\end{theor} 
\begin{proof}
If $d=1$, then  $V=\mathrm{Susp}(X,f,d)\cong X$.  Hence, $V$ is rigid. By Lemma~\ref{nod}, the variety $W$ is also rigid.

Let us consider the case $d>1$. By Proposition~\ref{dk}, there exists a rigid variety~$X$ such that $V=\mathrm{Susp}(X,f,d)$ is not rigid. Moreover,
$$V=\mathrm{Spec}(\KK[X][y]/\langle y^d-f\rangle)\subset \KK\times X,$$ and from the proof of Proposition~\ref{dk} there exists a locally nilpotent derivation $\partial$ on $\KK[V]$ such that $\partial(y)=0$. If we consider 
$$W=\mathrm{Susp}(X,f,k_1,\ldots,k_m)=\mathrm{Susp}(V,y,\frac{k_1}{d},\ldots,\frac{k_m}{d}),$$
then $\partial$ induces a nonzero LND on $\KK[W]$.
\end{proof}

\section{Automorphism group of a rigid $m$-suspension}

As before, let $W=\mathrm{Susp}(X,f,k_1,\ldots,k_m)$ and set $d=\gcd(k_1,\ldots, k_m)$. Then there is a natural action of an~$(m-1)$-dimensional torus $T\cong(\KK^\times)^{m-1}$ on it. If 
$$W=\mathrm{Spec}(\KK[X][y_1, \ldots, y_m]/\langle y_1^{k_1}\cdots y_m^{k_m}-f\rangle)\subset \KK^m\times X,$$ then $(t_1, \ldots, t_m) \cdot (y_1,\ldots,y_m)  =
  (t_1^{\frac{k_m}{d}}y_1, \ t_2^{\frac{k_m}{d}}y_2,\ \ldots,\ t_{m-1}^{\frac{k_m}{d}}y_{m-1}, 
\ t_1^{-\frac{k_1}{d}} \cdots t_{m-1}^{-\frac{k_{m-1}}{d}}y_m)$.
The action of the torus on elements of $\KK[X]$ is trivial.
We obtain that~$T\cong(\KK^\times)^{m-1}$ is a subtorus in the automorphism group $\Aut(W)$. 
\begin{lem}\label{lelem}
Let $m\geq 2$.
The functions
$y_1,\ldots, y_m$ are irreducible isolated $T$-semi-invariants. Moreover, each isolated $T$-semi-invariant has the form $\gamma \cdot  y_j$ for some~$\gamma\in\KK[W]^\times$ and $1\leq j\leq m$.
\end{lem}
\begin{proof}
Consider the system of generators $y_1,\ldots, y_m, \ z_1,\ldots, z_r$ of $\KK[W]$, where~$z_j$'s are  generators of $\KK[X]$. 
By definition, all functions in this system are $T$-semi-invariants. Let $e_1,\ldots, e_{m-1}$ be the standard basis in $\mathbb{Z}^{m-1}$. The weight of $y_i$ is~$\frac{k_m}{d}e_i$ for $1\leq i\leq m-1$. The weight of $y_m$ is~$\left(-\frac{k_1}{d},\ldots, -\frac{k_m}{d}\right)$. The weights of~$z_j$'s are equal to zero. 
We define 
\begin{align*}
    & \alpha_i(x_1,\ldots,x_{m-1})  =x_i \ \text{ for } \  1\leq i\leq m-1,\\
    & \alpha_m(x_1,\ldots,x_{m-1}) =-(x_1+\ldots+x_{m-1}).
\end{align*}
Using these functions, we obtain from Proposition~\ref{dsi}$(i)$ that all $y_i$'s are irreducible isolated $T$-semi-invariants. It also follows from Proposition~\ref{dsi}$(ii)$ that all irreducible isolated $T$-semi-invariants are of the form $\gamma \cdot y_j$, where $\gamma\in\KK[W]^\times$. 
\end{proof}

Let $X$ be a variety and $f\in\KK[X]$. Then let us denote by $\Aut(X,f)$ all automorphisms of $X$ such that $f$ is a semi-invariant for these automorphisms, i.e.,
$$
\Aut(X,f)=\{\varphi\in \Aut(X)\mid \varphi(f)=\lambda f, \ \lambda\in\KK^\times\}.
$$
Assume $k_1,\ldots,k_m$ are positive integers with $\gcd(k_1,\ldots,k_m)=1$. We define the embedding $\eta$ of $\Aut(X,f)$ into $\Aut(W)$, where $W=\mathrm{Susp}(X,f,k_1,\ldots,k_m)$. 

Suppose that $u_1,\ldots, u_k$ are such integers that $u_1k_1+\ldots+u_mk_m=1$. Now let~$\varphi\in \Aut(X,f)$, then $\varphi(f)=\lambda f$, $\lambda\in\KK^\times$. Let us define 
\begin{align*}
\eta(\varphi)(g) & =\varphi(g) \ \text{ for all }\  g\in \KK[X], \\
\eta(\varphi)(y_i) & =\lambda^{u_i}y_i \ \text{ for } \ i = 1, \ldots, m.
\end{align*}
We can easily see that $\eta(\Aut(X,f))$ commutes with the torus $T$ defined above. Thus, we get that $\eta(\Aut(X,f))\times T$ is isomorphic to $\Aut(X,f)\times T$ in $\Aut(W)$.

\begin{theor}\label{dva}
Let $X$ be rigid and $k_1,\ldots,k_m$ be positive with $\gcd(k_1,\ldots,k_m)=1$. For $W=\mathrm{Susp}(X,f,k_1,\ldots,k_m)$ we have 
\begin{enumerate}
    \item[$(i)$]  The neutral component $\Aut(W)^0$ is isomorphic to $\Aut(X,f)^0\times T$.
    \item[$(ii)$] When $\KK[X]^\times=\KK^\times$, then $\Aut(W)$ is a finite extension of $\eta(\Aut(X,f))\times T$.
\end{enumerate}
\end{theor}
\begin{proof}
$(i)$. By Theorem~\ref{perv}, $W$ is rigid. Hence, by \cite[Theorem 1]{AG}, there exists a unique maximal torus $\mathbb{T}$ in $\Aut(W)$. We have $T\subset  \mathbb{T}$. By Lemma~\ref{lelem}, each $y_i$ is an isolated irreducible $T$-semi-invariant. Since $\mathbb{T}$ is a commutative group, $y_i$ is an isolated irreducible $\mathbb{T}$-semi-invariant. By Lemma~\ref{finas}, the number~$r$ of association classes of isolated irreducible $\mathbb{T}$-semi-invariants is finite. Since every automorphism of $W$ permutes these classes, we obtain a homomorphism $\xi\colon \Aut(W)\rightarrow \mathrm{S}_r$.  Since~$S_r$ is a finite group, the neutral component $\Aut(W)^0$ is contained in the kernel of $\xi$. If $\varphi\in \Ker\xi$, then $\varphi$ multiplies all $\mathbb{T}$-semi-invariants by invertible functions. Each invertible function $h$ is not divisible by any $y_i$. Hence, the $T$-weight of $h$ is zero. Hence, $h$ belongs to the algebra of $T$-invariants $\KK[W]^T=\KK[X]$. It follows that $T$ is a normal subgroup of $\Ker\xi$. Therefore, $\KK[W]^T=\KK[X]$ is invariant under $\Ker\xi$, and hence $\KK[X]\subseteq \KK[W]$ is invariant under $\Aut(W)^0$. 

We obtain the homomorphism $\tau\colon\Ker\xi\rightarrow \left(\KK[X]^\times\right)^m$. Since the neutral component of $\left(\KK[X]^\times\right)^m$ is equal to $\left(\KK^\times\right)^m$, the image $\tau(\Aut(W)^0)$ acts on $y_i$ by multiplying each $y_i$ by a constant. Therefore, every $\psi\in \Aut(W)^0$ multiplies $f$ by a constant. Then we can take $t\in T$ and $\zeta\in \eta(\Aut(X,f))$ such that $\zeta\circ t\circ\psi$ maps every $y_i$ into $y_i$. Since $\KK[X]$ is invariant under each multiple, $\zeta\circ t\circ\psi\in \eta(\Aut(X,f))$. Thus, $\psi\in \eta(\Aut(X,f))\times T$. That is, $\psi\in \eta(\Aut(X,f)^0)\times T$.
 
$(ii)$. As we saw above, each element of $\Ker\xi$ multiplies each $y_i$ by an element of $\KK[X]^\times$. If $\KK[X]^\times=\KK^\times$, then each element $\varphi$ of $\Ker\xi$ multiplies each $y_i$ by a constant. Then, as above, we can prove that $\varphi\in \eta(\Aut(X,f))\times T$. 

\end{proof}

A direct corollary of Theorem~\ref{dva} is the following statement. 
\begin{cor}\label{clcv}
Suppose $X$ is rigid and the neutral component of its authomorphism group $\Aut(X)^0$ is a torus. Let $W=\mathrm{Susp}(X,f,k_1,\ldots,k_m)$, $m\geq 2$,  $\gcd(k_1,\ldots,k_m)=1$. Then $\Aut(W)^0$ is a torus. Moreover, if $\Aut(X)$ is a finite extension of the torus  $\Aut(X)^0$ and $X$ does not admits any non-constant invertible functions, then $\Aut(W)$ is a finite extension of the torus  $\Aut(W)^0$.
\end{cor}

\section{Automorphisms of rigid varieties\\ with a torus action of complexity $\leq 1$}

Recall that an irreducible variety $X$ is called \emph{toric} if it admits an action of an algebraic torus $T$ with an open orbit. It can be assumed that $\dim T=\dim X$. Toric varieties are often assumed to be normal by definition. But we allow that~$X$ can be non-normal. Moreover, it is proved in \cite{AKZ} that a normal toric variety without non-constant invertible function is flexible, which means that it is never rigid. In \cite[Theorem~2]{BG}, a criterion that a non-normal toric variety is rigid is obtained. In \cite[Theorem~3]{BG} and \cite[Remark~1]{BG}, it is proved that $\Aut(X)^0=T$, and if $X$ does not admit non-constant invertible functions, then $\Aut(X)$ is a finite extension of $T$. 
\begin{re}
 Note that this can be proved using our technique of isolated irreducible semi-invariants. Indeed, to every toric variety corresponds a finitely generated submonoid $P$ in the character group $\mathfrak{X}(T)\cong\mathbb{Z}^n$, which is a weight submonoid.
That is, $\KK[X]=\oplus_{m\in P}\KK\chi^m$, where $\chi^m$ is the character corresponding to~$m\in\mathfrak{X}(T)$.  This group generates the cone $\sigma^\vee\subseteq V\cong\mathbb{Q}^n=\mathfrak{X}(T)\otimes_\mathbb{Z}\mathbb{Q}$. A variety~$X$ admits non-constant invertible functions if and only if $\sigma^\vee$ contains a line. Let $L=\sigma^\vee\cap (-\sigma^\vee)$ be the largest subspace contained in $\sigma^\vee$.  We can consider the image $\gamma$ of $\sigma^\vee$ in the quotient space $V/L$. Then $\gamma$ is a pointed cone. It is easy to see that $\chi^m$ such that $m\in L$ is invertible and $\chi^m$ such that the images of $m$ are primitive vectors on the rays of $\gamma$ are isolated $T$-semi-invariants. Therefore, the algebra $\KK[X]$ is integer over the algebra $\mathcal{A}$ generated by all invertible functions and isolated irreducible $T$-semi-invariants. According to Corollary~\ref{cledct} we obtain the objective.   
\end{re}
Now let us consider a trinomial variety $X$ with the action of the torus $\mathbb{T}$ of complexity one, see Section~\ref{tv}. 
\begin{lem}\label{tip} 
Let $X$ be a trinomial variety of type 1. Suppose that there exists~$s$ such that $n_s>1$. Then $\KK[X]$ is an integer extension of $\mathcal{A}$, where $\mathcal{A}$ is the subalgebra generated by all isolated irreducible $\mathbb{T}$-semi-invariants.  
\end{lem}
\begin{proof}  
It follows from Lemma~\ref{lelem} that all $T_{ij}$ with $n_i>1$ are isolated irreducible $\mathbb{T}$-semi-invariants. Also, each $S_j$ is an isolated irreducible $\mathbb{T}$-semi-invariant. 
We can consider linear combinations of equations that give $X$ to get $T_i^{l_i}-T_s^{l_s}=a_s-a_i$. This shows that $\KK[X]$ is an integer extension of the algebra generated by all $T_{ij}$ with $n_i>1$ and all $S_k$. 
\end{proof}
\begin{prop}\label{prpp}
Suppose $X$ is a rigid trinomial variety with the action of the torus $\mathbb{T}$ of complexity one.
\begin{enumerate}
    \item[$(i)$] The neutral component $\Aut(X)^0$ is a torus $\overline{\mathbb{T}}$. 
    \item[$(ii)$] If $X$ does not admit non-constant invertible functions, then $\Aut(X)$ is a finite extension of $\overline{\mathbb{T}}$.  
\end{enumerate}
\end{prop}
\begin{proof}
Suppose $X$ is of type 1. If there exists $s$ such that $n_s>1$, then using Lemma~\ref{tip} and Corollary~\ref{nmmn} we have the objective. If all $n_i=1$, then $X$ is a surface and the statement of the proposition follows from \cite[Theorem~1.0.3]{PZ}. 

Suppose now that $X$ is of type 2. Then the action of $\mathbb{T}$ on $X$ is pointed, see Definition~\ref{ddppd}. If $\mathbb{T}$ is the maximal torus in $\Aut(X)$, therefore, by \cite[Proposition~1]{AG}, the authomorphism group $\Aut(X)$ is a finite extension of $\mathbb{T}$. If $\mathbb{T}$ is not the maximal torus, then the action of the maximal torus $\overline{\mathbb{T}}$ is also pointed and~$\Aut(X)$ is a finite extension of $\overline{\mathbb{T}}$.
\end{proof}

\begin{theor}\label{vteo}
Suppose $X$ is a rigid normal rational irreducible affine variety with only constant invertible regular functions, finitely generated divisor class group, and admitting a complexity one action of a torus $\mathbb{T}$. 
Then the neutral component~$\Aut(X)^0$ is a torus. 
\end{theor}
\begin{proof}
It is proved in~\cite[Corollary~1.9]{H-W}, that any normal rational irreducible affine variety $X$ with only constant invertible functions, finitely generated divisor class group, and admiting a torus action of complexity one can be canonically obtained by a categorical quotient of a trinomial variety $\overline{X}$ by an action of a quasitorus~$H\subseteq \mathbb{H}$, where $\mathbb{H}$ is the the centralizer of  $\mathbb{T}$ in $\mathrm{Aut}(X)$, which is also a quasitorus.

This realization of the variety $X$ by the categorical quotient is the Cox realization of this variety, see~\cite{ADHL}. By \cite[Theorem~4.2.3.2]{ADHL},  rigidity of $X$ implies rigidity of $\overline{X}$. It is proved in~\cite{AG1} that there is the following exact sequence of groups: 
$$
1\rightarrow H\rightarrow N_{\mathrm{Aut}(\overline{X})}(H) \rightarrow \mathrm{Aut}(X)\rightarrow 1.
$$
Here, $N_{\mathrm{Aut}(\overline{X})}(H)$ is the normalizer of $H$ in the automorphism group of $\overline{X}$. By Proposition~\ref{prpp}, the neutral component of $ \Aut(\overline{X})$ is a torus. Since we have $N_{\mathrm{Aut}(\overline{X})}(H)\subseteq \Aut(\overline{X})$, then the neutral component of $N_{\mathrm{Aut}(\overline{X})}(H)$ is a torus. Therefore, $\Aut(X)^0$ is a torus.
\end{proof}

\begin{re}
If $\overline{X}$ does not admit any non-constant invertible functions, then~$\Aut(X)$ is a finite extension of the torus $\Aut(X)^0$. However, this is not equivalent to the fact that $X$ does not admit non-constant invertible functions. Indeed, let us consider $\overline{X}=\mathbb{V}(x^2y^2-z^2w^2-1)$, then $xy\pm yz$ are invertible functions on  $\overline{X}$. If $X=\overline{X}/\mathbb{Z}_2$, where $\mathbb{Z}_2$ acts by multiplication of $x$ by $\pm1$, then $X$ does not admit any invertible functions, see \cite[Theorem~1.2(i)]{H-W}. 

An interesting question is whether it follows from the conditions of Theorem~\ref{vteo} with the additional assumption that $\KK[X]^\times=\KK^\times$ that $\Aut(X)$ is a finite extension of the torus.
\end{re}

\end{document}